\documentclass{amsart}
\usepackage{tikz}

\usepackage{amsmath,mathdots}
\usepackage{amssymb}
\usepackage{mathrsfs}
\usepackage[all]{xy}

\newtheorem{theorem}{Theorem}[section]
\newtheorem{claim}[theorem]{Claim}

\newtheorem{lemma}[theorem]{Lemma}

\newtheorem{proposition}[theorem]{Proposition}
\newtheorem{corollary}[theorem]{Corollary}

\newtheorem{fact}[theorem]{Fact}
\newtheorem*{thm.2.4}{Theorem \ref
{GlavinMainTheorem}}
\newtheorem*{cor2.23}{Corollary \ref{compactQpoint}}
\newtheorem*{thm.316}{Theorem \ref{thm.3.16}}
\newtheorem*{thm.317}{Theorem \ref{thm.3.17}}
\newtheorem*{cor.5.6}{Theorem \ref{ProductCor}}
\newtheorem*{cor.5.7}{Theorem \ref{PowerCor}}
\newtheorem*{thm6.10}{Theorem \ref{thm.bgclosedclass}}
\newtheorem*{thm.7.2}{Theorem
 \ref{analyse}}
\newtheorem*{thm.7.9}{Theorem \ref{Kunen-Paris-Model}}
\newtheorem*{thm.UA}{Theorem \ref{UACor}} 

\theoremstyle{definition}
\newtheorem{definition}[theorem]{Definition}
\newtheorem{example}[theorem]{Example}

\newtheorem{question}[theorem]{Question}

\theoremstyle{remark}
\newtheorem{remark}[theorem]{Remark}

\def\l{{\langle}}
\def\r{{\rangle}}

\newcount\skewfactor
\def\mathunderaccent#1#2 {\let\theaccent#1\skewfactor#2
\mathpalette\putaccentunder}
\def\putaccentunder#1#2{\oalign{$#1#2$\crcr\hidewidth
\vbox to.2ex{\hbox{$#1\skew\skewfactor\theaccent{}$}\vss}\hidewidth}}

\def\smallbox#1{\leavevmode\thinspace\hbox{\vrule\vtop{\vbox
   {\hrule\kern1pt\hbox{\vphantom{\tt/}\thinspace{\tt#1}\thinspace}}
   \kern1pt\hrule}\vrule}\thinspace}

\newcommand{\FIN}{\mathrm{FIN}}

\newcommand{\ra}{\rightarrow}
\newcommand{\al}{\alpha}
\newcommand{\om}{\omega}
\newcommand{\sse}{\subseteq}

\newcommand{\fin}{\mathrm{fin}}

\title[Tukey of Fubini]{On  the Tukey types of Fubini products}
\date{\today}

\author{Tom Benhamou}
\address[Benhamou]{Department of Mathematics, Rutgers University, ,110 Frelinghuysen Road Piscataway, New Jersey 08854-8019}
\thanks{The research of the first author was supported by the National Science Foundation under Grant
DMS-2346680.}

\email{tom.benhamou@rutgers.edu}

\author{Natasha Dobrinen}

\address[Dobrinen]{Department of Mathematics, University of Notre Dame, Notre Dame, IN 46556, USA}
\email{ndobrine@nd.edu}
\thanks{The  research of the second author was supported by  National Science Foundation Grant DMS-2300896.}


\begin{document}
\let\labeloriginal\label
\let\reforiginal\ref
\def\ref#1{\reforiginal{#1}}
\def\label#1{\labeloriginal{#1}}
\maketitle
\begin{abstract}
We extend the class of ultrafilters $U$ over countable sets for which $U\cdot U\equiv_T U$,
extending several results from \cite{Dobrinen/Todorcevic11}.
In particular, we prove that 
for each countable ordinal $\al\ge 2$, the generic ultrafilter $G_\al$ forced by $P(\om^\alpha)/\fin^{\otimes\alpha}$  satisfy $G_\al\cdot G_\al\equiv_T G_\al$. This answers a question posed in \cite[Question 43]{Dobrinen/Todorcevic11}. Additionally, we establish  that Milliken-Taylor ultrafilters possess the property that $U\cdot U\equiv_T U$.
\end{abstract}

\section{Introduction}
The Tukey order of partially ordered sets finds its origins in the notion of Moore-Smith convergence \cite{Tukey40}, which generalizes the usual meaning of convergence of sequence to  \textit{net}, allowing to enlarge the class of topological spaces for which continuity is equivalent to continuity in the sequential sense. Formally, given two posets, $(P,\leq_P)$ and $(Q,\leq_Q)$ we say that $(P,\leq_P)\leq_T (Q,\leq_Q)$ if there is map $f:Q\rightarrow P$, which is cofinal, namely, $f''B$ is cofinal in $P$ whenever $B\subseteq Q$ is cofinal. Schmidt \cite{Schmidt55} observed that this is equivalent to having a map $f:P\rightarrow Q$, which is unbounded, namely, $f''\mathcal{A}$ is unbounded in $Q$ whenever $\mathcal{A}\subseteq P$ is unbounded in $P$.
We say that $P$ and $Q$ are {\em Tukey equivalent},
and write
$P\equiv_T Q$,
if $P\leq_T Q$ and $Q\leq_T P$;  the equivalence class $[P]_T$ is called the Tukey type or cofinal type of $P$.
It turns out that the Tukey order restricted to posets $(U,\supseteq)$, where $U$ is an ultrafilter, has a close relation to \textit{ultranets} and has been studied extensively on $\omega$ by Blass, Dobrinen,  Kuzeljevic, Milovich, Raghavan, Shelah, Todorcevic, Verner, and others (see for instance 
\cite{Blass/Dobrinen/Raghavan15, DobrinenJSL15, Dobrinen/Mijares/Trujillo14,Dobrinen/Todorcevic11,
Kuzeljevic/Raghavan18,Milovich08,
Raghavan/Shelah17,Raghavan/Verner19}).
Recently, the authors extended this investigation to the realm of large cardinals where they considered the Tukey order on $\sigma$-complete ultrafilters over a measurable cardinal $\kappa$  in \cite{TomNatasha}. On ultrafilters, the Tukey order is determined by functions which are (weakly) monotone\footnote{That is, $a\leq b\Rightarrow f(a)\leq f(b)$.} and have cofinal images. For this reason, the Tukey order is the order one expects to use when comparing the cofinality of ultrafilters. 
We refer the reader to \cite{DobrinenTukeySurvey15}
and \cite{DobrinenSEALS}  for surveys  of  the subject.

In this paper, we investigate the connection between the Tukey type of an ultrafilter $U$ and the Tukey type  of its Fubini product with itself, $U\cdot U$. 
It is easy to see that $U\le_T U\cdot U$ for every ultrafilter $U$.  The question is, for which ultrafilters does $U\equiv_T U\cdot U$ hold?
Already Dobrinen and Todorcevic \cite{Dobrinen/Todorcevic11} and Milovich \cite{Milovich12}, provide a significant understanding of the relation between these two Tukey types: Dobrinen and Todorcevic proved that whenever $U$ is a rapid $p$-point ultrafilter over $\omega$,  then $U\cdot U\equiv_T U$; Milovich proved that  $U\cdot U\cdot U\equiv_T U\cdot U$ for every nonprincipal ultrafilter $U$.
In contrast, trivial examples of
  ultrafilters $U$ with the property that 
$U\cdot U\equiv_T U$
 are the so-called \textit{Tukey-top} ultrafilters,  those ultrafilters which are maximal in the Tukey order among all ultrafilters on $\omega$. Such an ultrafilter was constructed (in ZFC) by Isbell \cite{Isbell65}
 and Juh\'{a}sz \cite{Juhasz67}; we denote this ultrafilter by $\mathcal{U}_{\text{top}}$. Henceforth, we shall only focus on nonprincipal ultrafilters $U$ such that $U<_T\mathcal{U}_{\text{top}}$.

 Dobrinen and Todorcevic also proved that for $p$-points, $U\cdot U\equiv_T U$ is equivalent to $U$ being Tukey above $(\omega^\omega,\leq)$, where $\leq$ refers to the everywhere domination order of functions.
 Furthermore, they provided an example of a p-point $U$ which is not Tukey above $\omega^\omega$, and in particular satisfying $U<_T U\cdot U$.
 They also 
 asked 
 whether, besides the  Tukey-top ultrafilters, the class of ultrafilters which are Tukey  equivalent
 to their Fubini product is a subclass of the class of basically generated ultrafilters.


\begin{question}{\cite[Q.43]{Dobrinen/Todorcevic11}}\label{Question}
    Does $U\cdot U\equiv_T U<\mathcal{U}_{\text{top}}$ imply that $U$ is basically generated?
\end{question}
Note that on measurable cardinals, the situation is quite different, as  every $\kappa$-complete ultrafilter $U$ over $\kappa$ satisfies that $U\cdot U\equiv_T U$ \cite[Thm 5.6]{TomNatasha}.

In this paper, we provide a negative answer to this question by analyzing the Tukey type of the Fubini powers of generic ultrafilters obtained by $\sigma$-closed  forcings of the form $P(X)/I$.
We define the {\em $I$-pseudo intersection property}
(Definition \ref{def.Ipip})
as a way of abstracting the notion of p-point,
and use it to extend work in \cite{Dobrinen/Todorcevic11} and \cite{Milovich12}
to provide an  abstract condition  which guarantees that $U\cdot U\equiv_T U$
(Corollary \ref{Cor: Sufficient condition for product}).
We  then provide an equivalent formulation for $U$ being Tukey equivalent to itself 
(Theorem \ref{thm.main}), generalizing \cite[Thm.\ 35]{Dobrinen/Todorcevic11}.

In Section 2, we  apply Theorem \ref{thm.main}
to the  forcing 
$P(\om\times\om)/I$, where 
$I=\fin\otimes\fin:=\fin^{\otimes 2}$ over $\omega\times\omega$. 
This forcing was first investigated  by Szym\'{a}nski and Zhou \cite{Szymanski/Zhou} 
and later by many others.
We show that any generic ultrafilter  $G_2$ for $P(\omega\times\omega)/\fin^{\otimes 2}$  has the property that $G_2\cdot G_2\equiv_T G_2$.  Blass, Dobrinen and Raghavan \cite{Blass/Dobrinen/Raghavan15} showed that  $G_2$ 
is not  basically generated (so in particular is not a p-point) but also 
is not Tukey-top. 
 Dobrinen  proved in \cite{DobrinenJSL15} that $G_2$ is in fact the Tukey immediate successor of its projected Ramsey ultrafilter, so at the same level as a 
weakly Ramsey ultrafilter (see \cite{Dobrinen/Todorcevic14}) in the Tukey hierarchy.

In Section 3, we investigate ultrafilters  $G_{\al}$ obtained by forcing $P(\om^\alpha)/\fin^{\otimes\alpha}$, for any $2\le \alpha<\omega_1$. We analyze the Tukey type of $\fin^{\otimes\alpha}$ and prove that, for all $\al<\om_1$, $G_{\al}\equiv_TG_{\al}\cdot G_{\al}$
(Theorem \ref{thm.finalpha}).
Such ultrafilters  are not p-points and also not basically generated,
and each 
$G_\alpha$ Rudin-Keisler projects onto generic ultrafilters $G_\beta$ for $\beta<\alpha$. 
By results of Dobrinen \cite{DobrinenJSL15,DobrinenJML16, DobrinenInPrep}, such ultrafilters are  non-Tukey top, and moreover, are quite low in the Tukey hierarchy.
For instance,
for $ k<\omega$, the sequence $\l G_n\mid  n\leq k\r$ form an exact Tukey chain in the sense that if $V\leq_TG_k$, then there is $ n\leq k$, $G_{n}\equiv_T V$. Related results holds for $\omega\leq \alpha<\omega_1$.

Finally, 
in Section 4, we prove that 
 if $U$ is a Milliken-Taylor ultrafilter then also $U$ satisfies our equivalent condition and therefore $U\cdot U\equiv_T U$
 (Theorem \ref{thm.MillikenTaylor}).
Milliken-Taylor ultrafilters as well as those forced by $P(\om^{\al})/\fin^{\otimes\al}$ are not basically generated and in particular p-points, providing examples which answer Question \ref{Question} in the negative.

One question which our results are relevant for but remains open, is  whether $U\cdot V\equiv_T V\cdot U$ for any ultrafilter $U,V$ over $\omega$.
Corollaries \ref{cor.1.9} and \ref{cor.1.10} provide some progress on this question.
This again contrasts with ultrafilters on 
 a measurable cardinal $\kappa$:  work of the authors in  \cite{TomNatasha} showed that if $U,V$ are $\kappa$-complete ultrafilters over $\kappa$, then $U\cdot V\equiv_T V\cdot U$. The proof essentially uses the well-foundedness of ultrapowers by $\kappa$-complete ultrafilters and therefore does not apply for ultrafilters over $\omega$.
This is discussed in Section 1.

\subsection{Notation} $[X]^{<\lambda}$ denotes the set of all subsets of $X$ of cardinality less than $\lambda$. Let $\fin=[\omega]^{<\omega}$, and $\FIN=\fin\setminus\{\emptyset\}$.
For a collection of sets $(P_i)_{i\in I}$ we let $\prod_{i\in I}P_i=\{f:I\rightarrow \bigcup_{i\in I}P_i\mid \forall i,\, f(i)\in P_i\}$. Given a set $X\subseteq \omega$, such that $|X|=\alpha\leq\omega$, we denote by $\l X(\beta)\mid \beta<\alpha\r$ be the increasing enumeration of $X$. Given a function $f:A\rightarrow B$, for $X\subseteq A$ we let $f''X=\{f(x)\mid x\in X\}$ and for $Y\subseteq B$ we let $f^{-1}Y=\{x\in X\mid f(x)\in Y\}$. Given sets $\{A_i\mid i\in I\}$ we denote by $\biguplus_{i\in I}A_i$ the union of the $A_i$'s when the sets $A_i$ are pairwise disjoint.

\section{The Tukey type of a Fubini product}

Given $\mathcal{A}\subseteq P(X)$, we set $\mathcal{A}^*=\{X\setminus A\mid A\in\mathcal{A}\}$. For a filter $F$ over $X$,
we denote the \textit{dual ideal} by $F^*$,
and given an ideal $I$ we denote the \textit{dual filter} by $I^*$. 
The following fact is easy to verify:
\begin{fact}\label{fact: dual ideal tukey}
    For any filter $F$, $(F,\supseteq)\equiv_T(F^*,\subseteq)$.
\end{fact}
An ultrafilter over $X$ is a filter $U$ such that for every $A\in P(X)$, either $A\in U$ or $X\setminus A\in U$. So for ultrafilters we have that $U^*=P(X)\setminus U$. As the title of this section indicates, we are interested in the Fubini product of ultrafilters:
\begin{definition}
    Suppose that $U$ is a filter over $X$ and for each $x\in X$, $U_x$ is a filter over $Y_x$. We denote by $\sum_UU_x$ the filter over $\bigcup_{x\in X}\{x\}\times Y_x$, defined by
    $$A\in \sum_UU_x\text{ if and only if }\{x\in X\mid (A)_x\in U_x\}\in U$$
    where $(A)_x=\{y\in Y_x\mid \l x,y\r\in A\}$.  If for every $x$, $U_x=V$ for some fixed $V$ over a set $Y$, then $U\cdot V$ is defined as $\sum_UV$, which is a filter over $X\times Y$. $U^2$ denotes the filter $U\cdot U$ over $X\times X$. 
\end{definition}
It is well known that if $U$ and each $U_x$ are ultrafilters, then also $\sum_UU_x$ is an ultrafilter (see for example see \cite{Blass1970}). 

\begin{fact}\label{fact: reduce to ultrafilters on omega}
    If $U,V$ are filters over countable sets $X,Y$ respectively, then $U\equiv_{RK}U'$, $V\equiv_{RK}V'$ for some filters $U',V'$ over $\omega$ and $U\cdot V\equiv_{RK}U'\cdot V'$.
\end{fact}

Given posets $(P_i,\leq_i)_{i\in I}$ we let $\prod_{i\in I}(P_i,\leq_i)$ be the ordered set $\prod_{i\in I}P_i$ with the pointwise order derived from the orders $\leq_i$. We call this order on $\prod_{i\in I}P_i$, the \textit{everywhere domination order}.
We will omit the order when it is the natural order. In particular, any ordinal $\alpha$ is ordered naturally by $\in$, $\omega^\omega=\prod_{n<\omega}\omega$ is ordered by everywhere domination, and $\prod_{n<\omega}\omega^\omega$ is the everywhere domination   order where each $\omega^\omega$ is again ordered by everywhere domination. 

    \begin{fact}\label{fact: carinality of index set of product}
        If $A,B$ are any sets with the same cardinality and $(P,\preceq)$ is an ordered set, then  $\prod_{a\in A}(P,\preceq)$ and $\prod_{b\in B}(P,\preceq)$ equipped with the everywhere domination orders are order isomorphic.
    \end{fact}
    The next theorem of Dobrinen and Todorcevic  provides an upper bound for the Fubini product of ultrafilters via the Cartesian  product.
\begin{theorem}[Dobrinen-Todorcevic, Thm.\ 32, \cite{Dobrinen/Todorcevic11}]\label{Theorem: basic bound for product}
    $\sum_UU_x\leq_T U\times \prod_{x\in X}U_x$. In particular, $U\cdot V\leq_T U\times \prod_{x\in X}V$ and $U\cdot U\leq_T \prod_{x\in X}U$.
\end{theorem}

Towards answering a question from \cite{Dobrinen/Todorcevic11},
Milovich  improved the previous theorem 
over $\omega$. Let us give  a slight variation of his proof, for we will need it later: 
\begin{proposition}[Milovich, Lemma 5.1, \cite{Milovich12}]
\label{Prop: equivalence with omega product}
For filters $U,V$ over countable sets $X,Y$ (resp.), $U\cdot V\equiv_T U\times\prod_{x\in X} V$. In particular $U\cdot U\equiv_T\prod_{x\in X} U$.
\end{proposition}
\begin{remark}
    We cannot prove  in general that $\sum_UU_n\equiv_TU\times\prod_{n<\omega}U_n$. For example, if $U$ is such that $U\cdot U\equiv_TU$ (e.g., $U$ is Ramsey) and $U_0$ is Tukey-top while $U_n=U$ for every $n>0$, then $\sum_UU_n=U\cdot U\equiv_TU<_TU_0$ but $\prod_{n<\omega}U_n\geq_T U_0$, so we have $\sum_UU_n\not\equiv_T U\times \prod_{n<\omega}U_n$. 
    Similar examples can be constructed even when all the $U_n$'s are distinct, for example, requiring that for every $n>0$, $U_n\leq_T U$ for some $U$ such that $U\cdot U\equiv_T U$.
    Such examples are constructed in this paper in Section 3: Take $U$ to be a generic ultrafilter for $P(\omega^\omega)/\fin^{\otimes\omega}$, and $U_n$ the Rudin-Keisler projection of $U$ to a generic on $P(\omega^n)/\fin^{\otimes n}$.
    
\end{remark}
\begin{proof}
    The proof of Theorem 
    \ref{Theorem: basic bound for product}
    does not use the fact that the partial orders are ultrafilters, and we have that $U\cdot V\leq_T U\times\prod_{x\in X} V$. (Indeed, the map $F:U\times\prod_{x\in X}V\rightarrow U\cdot V$ defined by $F(A,\l Y_x\mid x\in X\r)=\bigcup_{x\in A}\{x\}\times Y_x$ is monotone and cofinal). For the other direction, by Facts \ref{fact: reduce to ultrafilters on omega} and \ref{fact: carinality of index set of product},   we may assume that $X=Y=\omega$. Let us define a cofinal map from a cofinal subset of $U\cdot V$ to $U\times \prod_{n<\omega}V$. Consider the collection $\mathcal{X}\subseteq U\cdot V$ of all $A\subseteq \omega\times \omega$ such that:
    \begin{enumerate}
        \item for all $n<\omega$, either $(A)_n=\emptyset$ or $(A)_n\in V$.
        \item $\pi''A\in U$ and for all $n_1,n_2\in\pi_1''A$, if $n_1<n_2$ then $(A)_{n_2}\subseteq (A)_{n_1}$.
    \end{enumerate}
    It is not hard to prove that $\mathcal{X}$ is a filter base for $U\cdot V$. Let $F:\mathcal{X}\rightarrow U\times\prod_{n<\omega}V$ be defined by $F(A)=\l \pi''A,\l(A)_{(\pi_1''A)(n)}\mid n<\omega\r\r$ where $\l(\pi_1''A)(n)\mid n<\omega\r$ is the increasing enumeration of $\pi_1''A$. Let us prove that $F$ is monotone and cofinal. Suppose that $A,B \in \mathcal{X}$ are such that $A\subseteq B$. Then
    \begin{enumerate}
        \item [a.] $\pi''A\subseteq \pi'' B$;
        \item [b.] for every $n<\omega$, $(\pi''A)(n)\geq(\pi''B)(n)$;
        \item [c.] for every $m<\omega$, $(A)_m\subseteq (B)_m$.
    \end{enumerate} 
    By requirement $(2)$ of sets in $\mathcal{X}$, for every $n<\omega$,
    $$(A)_{(\pi''A)(n)}\subseteq (B)_{(\pi''A)(n)}\subseteq (B)_{(\pi''B)(n)}.$$
    It follows that $F(A)\geq F(B)$. To see that  $F$ is cofinal, let $\l B,\l B_n\mid n<\omega\r\r\in U\times\prod_{n<\omega}V$. Define $A=\bigcup_{n\in B}\{n\}\times (\bigcap_{m\leq n}B_m)$. Then $\pi''A=B$ and it is straightforward that $A\in\mathcal{X}$. We claim that for every $n$, $F(A)_n\subseteq B_n$. Indeed, $B(n)=(\pi''A)(n)\geq n$ and therefore $F(A)=\l B,\l A_n\mid n<\omega\r\r$ where $A_n=\bigcap_{m\leq B(n)}B_m\subseteq B_n$.  
\end{proof}
Milovich used this proposition to deduce the following, answering a question in \cite{Dobrinen/Todorcevic11}
and improving  a result in \cite{Dobrinen/Todorcevic11} which showed that (2) below holds if $F$ and $G$ are both rapid p-points:
\begin{theorem}[Milovich, Thms.\ 5.2, 5.4, \cite{Milovich12}]
    \begin{enumerate}
        \item For any filters $F,G$, \\
        $F\cdot G\cdot G\equiv_T F\cdot G$. In particular,  $F^{\otimes 2}\equiv_TF^{\otimes 3}$.
        \item If $F,G$ are $p$-filters, then $F\cdot G\equiv_TG\cdot F$.
    \end{enumerate}
\end{theorem}

For the rest of this section, let us derive several  corollaries and present new results.
\begin{corollary}\label{cor.1.9}
    Suppose that $V\cdot V\equiv_T V$.
    Then $U\cdot V\equiv_T U\times V$. 
    Moreover, if also    
    $U\cdot U\equiv_T U$, then $U\cdot V\equiv_T V\cdot U$.
\end{corollary}
\begin{proof}
    Since $V\cdot V\equiv_T V$, we have that $V\equiv_T \prod_{n<\omega}V$, hence $$U\cdot V\equiv_T U\times \prod_{n<\omega} V\equiv_TU\times V.$$ For the second part, it is clear that $U\times V\equiv_T V\times U$ and therefore if $V\cdot V\equiv_T V$ and $U\cdot U\equiv_T U$, then $U\cdot V\equiv_T V\cdot U$.
\end{proof}

\begin{corollary}\label{cor.1.10}
    For any ultrafilters $U,V$ on countable sets, $U\cdot V\cdot V\equiv_T U\times (V\cdot V)$. In particular $(U\cdot U)\cdot (V\cdot V)\equiv_T(V\cdot V)\cdot (U\cdot U)$.
\end{corollary}



As a consequence of Proposition \ref{Prop: equivalence with omega product}, $U\cdot U\equiv_T U$ if and only if $U\equiv_T\prod_{n<\omega}U$. Still, checking whether $U$ is Tukey equivalent to $\prod_{n<\omega}U$ is usually a non-trivial task. We provide below a simpler condition and explain how it generalizes some  results from \cite{Dobrinen/Todorcevic11}.
\begin{definition}\label{def.Ipip}
    Let $U$ be an ultrafilter over a countable set $X$ and $I\subseteq U^*$ an ideal on $X$. We say that {\em $U$ has the  $I$-pseudo intersection property}
    (abbreviated by $I$-p.i.p.) if for any sequence $\l A_n\mid n<\omega\r\subseteq U$ there is a set $A\in U$ such that for every $n<\omega$, $A_n\subseteq^I A$, namely, $A_n\setminus A\in I$.
    \end{definition}
This definition is a generalization of being a $p$-point as the following example suggests.
\begin{example}
    $U$ having  the $\fin$-p.i.p.\  is equivalent to $U$ being a $p$-point.
\end{example}
\begin{claim}\label{Claim: U is *U*}
    For every ultrafilter $U$ over $X$, $U^*$-p.i.p.\ holds.
\end{claim}
\begin{proof}
    For any sequence $\l A_n\mid n<\omega\r$ we  can always take $A=X\in U$, since then $A\setminus A_n= X\setminus A_n\in U^*$ by definition.
\end{proof}
\begin{proposition}
    If $U$ has the $I$-p.i.p., then $$U\cdot U\leq_T U\times\prod_{n<\omega}I$$
\end{proposition}
\begin{proof}
    Let $U$ be any ultrafilter. Then by Proposition \ref{Prop: equivalence with omega product}, $U\cdot U\equiv_T \prod_{n<\omega} U$. We claim that if $U$ has the $I$-p.i.p., then $\prod_{n<\omega}U\leq_T U\times \prod_{n<\omega} I$. Let $\l A_n\mid n<\omega\r\in\prod_{n<\omega}U$,  and choose $A\in U$ such that $A\setminus A_n\in I$, which exists by $I$-p.i.p. Define $F(\l A_n\mid n<\omega\r)=\l A,\l A\setminus A_n\mid n<\omega\r\r\in U\times \prod_{n<\omega}I$. We claim that $F$ is unbounded.  Indeed, suppose that $\mathcal{A}\subseteq \prod_{n<\omega}U$ and $F''\mathcal{A}$ is bounded by $\l A^*,\l X_n^*\mid n<\omega\r\r\in U\times\prod_{n<\omega}I$. Define $A^*_n=A^*\setminus X^*_n$. Note that  $I\subseteq U^*$ and $A^*\in U$ imply  that $\l A^*_n\mid n<\omega\r\in \prod_{n<\omega}U$. Let us show that this is a bound for $\mathcal{A}$. Let $\l A_n\mid n<\omega\r\in\mathcal{A}$. Then $F(\l A_n\mid n<\omega\r\in\mathcal{A})=\l A,\l A\setminus A_n\mid n<\omega\r\r\leq \l A^*,\l X^*_n\mid n<\omega\r\r$, namely, $A^*\subseteq A$ and $A\setminus A_n\subseteq X^*_n$. It follows that $A^*_n=A^*\setminus X^*_n\subseteq A\setminus (A\setminus A_n)=A\cap A_n\subseteq A_n$. Hence $\l A_n\mid n<\omega\r\leq \l A^*_n\mid n<\omega\r$, as desired.
\end{proof}
\begin{corollary}
    
    If $U$ is a $p$-point then $U\cdot U\leq_T U\times \prod_{n<\omega} \fin$.

\end{corollary}
\begin{fact}\label{fact.1.15}
    $(\fin,\subseteq)\equiv_T (\omega,\leq)$.
\end{fact}
\begin{proof}
    The collection of all sets of the form $\{0,...,n\}$ is cofinal in $\fin$ and is clearly order isomorphic to $\omega$.
\end{proof}
    Translating Fact  \ref{fact.1.15} to $\prod_{n<\omega}\fin$, we see that $\prod_{n<\omega}\fin\equiv_T(\omega^\omega,\leq)$,
where $\omega^\omega$ is the set of all functions $f:\omega\rightarrow\omega$ and the order $\leq$ refers to the everywhere domination order. We conclude that if $U$ is a $p$-point then $U\cdot U\leq_T U\times \omega^\omega$ (this is the important part of \cite[Thm. 33]{Dobrinen/Todorcevic11}). 
We can now derive the following sufficient condition:
\begin{corollary}\label{Cor: Sufficient condition for product}
    Let $U$ be an ultrafilter over a countable set $X$, $I\subseteq U^*$ an ideal on $X$ and \begin{enumerate}
        \item
        $U$ has the $I$-p.i.p., and 
        \item $\prod_{n<\omega}I\leq_T U$.
    \end{enumerate}
    Then $U\cdot U\equiv_TU$.
\end{corollary}
    \begin{proof}
        $U\leq_T U\cdot U\leq_T U\times \prod_{n<\omega}I\leq_T U\times U\equiv_TU$.
    \end{proof}
The above sufficient condition is in fact an equivalence:
        \begin{theorem}\label{thm.main}
            For every ultrafilter $U$ over a countable set $X$, the following are equivalent:
            \begin{enumerate}
                \item $U\cdot U\equiv_TU$.
                \item $\prod_{n<\omega} U\equiv_T U$.
                \item There is an ideal $I$ such that $I$-p.i.p.\ holds and $\prod_{n<\omega}I\leq_T U$.
        \end{enumerate} 
        \end{theorem}
\begin{proof}
    $(1)$ and $(2)$ are equivalent by Proposition \ref{Prop: equivalence with omega product}.  $(2)\Rightarrow (3)$ is trivial, taking $I=U^*$ and by Claim \ref{Claim: U is *U*}. Finally, $(3)\Rightarrow(1)$ follows from the previous corollary. 
\end{proof}
In particular, if $U$ is a $p$-point, the above  proposition provides the equivalence that $U\cdot U\equiv_TU$ if and only if $U\geq_T \omega^\omega$, recovering   \cite[Thm. 35]{Dobrinen/Todorcevic11}.

Recall that an ultrafilter $U$ over $\omega$ is \textit{rapid}  if for every increasing function $f:\omega\rightarrow\omega$ there is $X\in U$ such that for every $n<\omega$, $otp(X\cap f(n))\leq n$.
        \begin{fact}
            $U$ is rapid if and only if the following map is cofinal: $F:U\rightarrow \omega^\omega$ defined by $F(X)=\l X(n)\mid n<\omega\r$, where $X(n)$ is the $n^{\text{th}}$ element of $X$.  
        \end{fact}
        As a corollary, we obtain once more a result from \cite{Dobrinen/Todorcevic11}:
        \begin{corollary}
            If $U$ is a rapid $p$-point then $U\equiv_TU\cdot U$.
        \end{corollary}
       
        By taking ideals other than $\fin$, we will find ultrafilters that are not $p$-points but are Tukey equivalent  to their Fubini product.
        
        \section{The ideal $\fin\otimes \fin$}
        Let $I,J$ be ideals on $X,Y$ (resp.). We define the Fubini product of the ideals $I\otimes J$ over $X\times Y$: For $A\subseteq X\times Y$,
        $$A\in I\otimes J\text{ iff } \{x\in X\mid (A)_x\notin J\}\in I.$$
        We note that this is the dual operation of the Fubini product of filters:
        \begin{fact}
            For every two ideals $I,J$, $(I\otimes J)^*=I^*\cdot J^*$.
        \end{fact}
        Our main interest in this section is the ideal $\fin\otimes \fin$ on $\omega\times\omega$, which  is defined by
            $$X\in \fin\otimes \fin\text{ iff }\{n<\omega\mid (X)_n\text{ is infinite}\}\text{ is finite.}$$
        \begin{proposition}\label{Prop: Tukey class of Fin2}
            $(\fin\otimes \fin,\subseteq)\equiv_T\omega^\omega$, where on $\omega^\omega$ we consider the everywhere domination order. 
        \end{proposition}
        \begin{proof}
           By Fact \ref{fact: dual ideal tukey} and the previous fact, $(\fin\otimes\fin,\subseteq)\equiv_T((\fin\otimes\fin)^*,\supseteq) \equiv_T(\fin^*\cdot\fin^*,\supseteq)$. By Proposition \ref{Prop: equivalence with omega product},
           $$(\fin^*\cdot\fin^*,\supseteq)\equiv_T\prod_{n<\omega}\fin^*\equiv_T\prod_{n<\omega}\fin\equiv_T\prod_{n<\omega}\omega=\omega^\omega.$$           \end{proof}
       \begin{corollary}
            Suppose that $U$ is an ultrafilter over $\omega\times\omega$ such that $\fin\otimes\fin\subseteq U^*$, $\fin\otimes\fin$-p.i.p.\ holds for $U$, and $\prod_{n<\omega}(\omega^\omega,\leq)\leq_T U$. Then $U\cdot U\equiv_TU$.
        \end{corollary}
        \begin{proof}
            Apply Corollary \ref{Cor: Sufficient condition for product} for $I=\fin\otimes\fin$.
        \end{proof}
        The order $\prod_{n<\omega}\omega^\omega$ can be simplified:
        \begin{fact}\label{fact: taking power of omega}
        $\prod_{n<\omega}\omega^\omega$ is order isomorphic to $\omega^\omega$ and in particular $\prod_{n<\omega}\omega^\omega\equiv_T\omega^\omega$.
        \end{fact}
        \begin{proof}
            Take any partition of $\omega$ into infinitely many infinite sets $\l A_n\mid n<\omega\r$. Then any function $f:\omega\rightarrow\omega$ induces functions $\l f\restriction A_n\mid n<\omega\r\in \prod_{n<\omega}\omega^{A_n}$. Clearly, $\omega^{A_n}$ is isomorphic to $\omega^\omega$ by composing each function $f:A_n\rightarrow \omega$ with the inverse of the transitive collapse $\pi_n:A_n\rightarrow \omega$.
        \end{proof}
        We now look for conditions that  guarantee that $U\geq_T \omega^\omega$. One way, is to ensure that the Rudin-Keisler projection on the first coordinate is rapid:
        \begin{definition}
            Suppose that $U$ is an ultrafilter such that $\fin\otimes\fin\subseteq U^*$.  We say that $U$ is {\em $2$-rapid} if the ultrafilter $\pi_*(U)=\{X\subseteq \omega\mid \pi^{-1}X\in U\}$
            is a rapid ultrafilter on $\omega$, where $\pi:\omega\times\omega\rightarrow \omega$ is the projection to the first coordinate.
            
        \end{definition}
        \begin{corollary}\label{Corollary: sufficient for Fin2}
            Suppose that $U$ is an ultrafilter on $\omega\times\omega$ such that $\fin\otimes\fin\subseteq U^*$, and $U$ is 
            $\fin\otimes\fin$-p.i.p.\ and $2$-rapid. 
            Then $U\cdot U\equiv_T U$.
        \end{corollary}
        Given an ideal $I$ on a set $X$, an {\em $I$-positive} set is any set in $I^+:=P(X)\setminus I$. 
        The forcing $P(X)/I$ is forcing equivalent to $(I^+, \sse^{I})$, where  the pre-order is given by $X\sse^I Y$ iff $X\setminus Y\in I$.
        If $G\subseteq P(X)$ is $I^+$-generic over $V$, then $G$ is an ultrafilter for the algebra $P^V(X)$ (namely $(V,\in ,G)\models ``G$ is an ultrafilter") 
        and also $I\subseteq G^*$. 
        We will only be interested in the case where $X$ is a countable set and $P(X)/I$ is $\sigma$-closed. This is equivalent to the following property of $I$: we say that $I$ is a \textit{$p^+$-ideal} 
        if whenever $\l A_n\mid n<\omega\r$ is a $\subseteq$-decreasing sequence of $I$-positive sets, there is an $A\in I^+$ such that for every $n<\omega$, $A\setminus A_n\in I$. 

        Given that $P(X)/I$ is $\sigma$-closed, the forcing does not add new reals. Hence, if $G$ is $I^+$-generic over $V$ then $P^V(X)=P^{V[G]}(X)$  and thus, $G$ is an ultrafilter 
         over $X$ in $V[G]$.
         Clearly,  $\fin$ is a $p^+$-ideal,  and it is well known that the generic ultrafilter for $P(\omega)/\fin$ is selective (and therefore a $p$-point and rapid):

        \begin{fact}[Folklore]
        \label{fact: generic for P(omega)/FIN is selective}
            If $G$ is $P(\omega)/\fin$-generic over $V$,  then
             $G$ is a Ramsey ultrafilter in $V[G]$.
        \end{fact}
        In particular, $G$ is a rapid $p$-point.
        By results in \cite{Dobrinen/Todorcevic11}, $G\cdot G\equiv_T G<\mathcal{U}_{\text{top}}$, and in fact,  by results in \cite{Raghavan/Todorcevic12}, $G$ is Tukey-minimal among nonprincipal ultrafilters.
        However, this does not answer Question \ref{Question} as $G$ is a $p$-point and therefore basically generated.
         We will answer Question \ref{Question} below.

        Next, let us move to the forcing $P(\omega\times\omega)/\fin\otimes\fin$. This forcing was first considered in \cite{Szymanski/Zhou} and later in many papers including
        \cite{HrusaVerner}, \cite{Blass/Dobrinen/Raghavan15}, \cite{DobrinenJSL15}, \cite{DobrinenHathaway20}, and \cite{BST22}. Again,  it is not hard to see that $\fin\otimes\fin$ is a $p^+$-ideal. The following properties of the generic ultrafilter are due to Blass, Dobrinen and Raghavan \cite{Blass/Dobrinen/Raghavan15} and  Dobrinen \cite{DobrinenJSL15}:\begin{theorem}\label{theorem: the ultrafilter obtained by P(omega)/FIN2}
            Let $G$ be a $P(\omega\times\omega)/\fin\otimes\fin$-generic ultrafilter over $V$. Then:
        \begin{enumerate}
             
            \item $G$ is not Tukey top and is also not basically generated.
            \item $\pi_*(G)$ is $P(\omega)/\fin$-generic over $V$, where $\pi:\omega\times\omega\rightarrow\omega$ is the projection to the first coordinate. 
            \item 
           $G$ is the immediate successor of $\pi_*(G)$ in the Tukey order.
        \end{enumerate}
        \end{theorem}
         \begin{proof}
         For the convenience of the reader, we provide references to the proofs of the above:
             \begin{enumerate}
                 \item \cite[Thms. 47 and 60]{Blass/Dobrinen/Raghavan15}.
                 \item \cite[Prop. 30]{Blass/Dobrinen/Raghavan15} \item \cite[Thm. 6.2] {DobrinenJSL15}.
             \end{enumerate}
         \end{proof}

         Together with  Theorem \ref{theorem: the ultrafilter obtained by P(omega)/FIN2}, the following theorem provides an answer to Question \ref{Question}:
        \begin{theorem}
        Let $G$ be  a $P(\omega\times\omega)/\fin\otimes\fin$-generic ultrafilter over $V$. Then
            \begin{enumerate}
                \item $\fin\otimes\fin\subseteq G^*$.
                \item $G$ satisfies $\fin\otimes\fin$-p.i.p.
                \item $G$ is $2$-rapid.
                \item $G\cdot G\equiv_TG$.
            \end{enumerate}
        \end{theorem}
        \begin{proof}
              $(1)$ is trivial. To see $(2)$, the argument is the same as showing that the forcing is $\sigma$-closed. For the self-inclusion of this paper, let us provide an indirect proof assuming $\sigma$-closure. Let $\l X_n\mid n<\omega\r\subseteq G$. Since the forcing is $\sigma$-closed, $\l X_n\mid n<\omega\r\in V$. Let $X\in G$ be such that $X\Vdash \forall n, X_n\in \dot{G}$, then $X\leq X_n$. Otherwise, $X\setminus  X_n\in (\fin\otimes\fin)^+$ and then $(X\setminus X_n)\leq X$ is a stronger condition which forces that $X_n\notin \dot{G}$, contradiction. Hence for every $n<\omega$, $X\setminus X_n\in \fin\otimes\fin$. 

             For $(3)$, we apply the previous theorem clause $(4)$ to see that  $\pi_*(G)$ is generic for $P(\omega)/\fin$ and therefore rapid by Fact \ref{fact: generic for P(omega)/FIN is selective}. Finally, $(4)$ follows from $(1)-(3)$ and Corollary \ref{Corollary: sufficient for Fin2}.

        \end{proof}

        \section{Transfinite iterates of $\fin$}

In this section we obtain analoguous results to the ones from the previous section, but for ultrafilters with higher cofinal-type complexity. To do so, we will consider the  generic ultrafilters  $G_{\al}$ obtained by the  forcing $P(\omega^\alpha)/\fin^{\otimes\alpha}$, where $1\le \al<\om_1$
(see the paragraph following Theorem \ref{Theorem: Natasha's result JSL}). Such ultrafilters were investigated in 
 \cite{DobrinenJSL15} and in yet unpublished work \cite{DobrinenInPrep}. 
 We point out that for $2\le \al$, $G_{\al}$ is not a p-point and not basically generated; and for $\beta<\al$, there are natural Rudin-Keisler projections from $G_{\al}$ to $G_{\beta}$.
 
 \begin{theorem}[Dobrinen]\label{Theorem: Natasha's result JSL}
     Suppose that $1\le \al<\om_1$ and  $G_{\al}$ be a generic ultrafilter obtained by  forcing  with $P(\omega^{\al})/\fin^{\otimes\alpha}$ over $V$.
Then $G_{\al}$ is not Tukey top and also not basically generated.  Moreover,
\begin{enumerate}
\item  For each $1\le k<\om$, 
the collection of Tukey types of ultrafilters Tukey-reducible to $G_k$ forms a chain of length $k$ consisting exactly of Tukey types of $G_n$ for $1\le n\le k$. \cite[Thm.\ 6.2]{DobrinenJSL15}
\item 
For each $\om<\al<\om_1$,
the Tukey types of the $G_{\beta}$, $1\le \beta\le\al$ are all distinct  and form a chain, but there are actually $2^{\om}$ many Tukey types below $G_{\al}$.
\cite{DobrinenInPrep}
\end{enumerate}
 \end{theorem}


The following recursive definition of $\fin^{\otimes\al}$, for $2\le\al<\om_1$,  is well-known and has appeared in 
\cite{BST22},
\cite{DobrinenJSL15},
    \cite{DobrinenJML16},  and \cite{Kurilic15}. 
        \begin{enumerate}
            \item At successor steps, $\fin^{\otimes\alpha+1}=\fin\otimes\fin^{\otimes\alpha}$ is the ideal on $\omega^{\alpha+1}=\omega\times \omega^\alpha$; explicitly, 
$A\sse \om^{\al+1}$ is in $\fin^{\otimes\al+1}$  iff for all but finitely many $n$, $(A)_n\in\fin^{\otimes\al}$.
            \item For limit $\al<\omega_1$ we fix an increasing  sequence $\l \al_n\mid n<\omega\r$
            with $
            \sup_{n<\om}\al_n=\al$ and define $\fin^{\otimes\al}=\sum_{\fin}\fin^{\otimes\al_n}$ on $\omega^{\al}:=\biguplus_{n<\omega}\{n\}\times \omega^{\al_n}$; explicitly, 
         $A\sse \om^{\al}$ is in $\fin^{\otimes\al}$ iff for all but finitely many $n$, $(A)_n$ is in $\fin^{\otimes \al_n}$.

        \end{enumerate}

       The Rudin-Keisler order is defined as follows: Let $I,J$ be ideals on $X,Y$ respectively. We say that  $I\leq_{RK} J$ if there is a function $f:Y\rightarrow X$ such that $f_*(J)=I$, where $$f_*(J)=\{A\subseteq X\mid \pi^{-1}[X]\in J\}.$$
       It is well known that the Rudin-Keisler order implies the Tukey order.
        \begin{lemma}[Folklore]
  For  $1\le \beta\leq \al<\om_1$, we have 
      $(\fin^{\otimes\beta},\sse)\leq_{RK} (\fin^ {\otimes\al},\sse)$ and therefore $(\fin^{\otimes\beta},\sse)\leq_{T} (\fin^ {\otimes\al},\sse)$.
        \end{lemma}
        \begin{proof}
             By induction on $\al$. At successor steps, we define the projection to the second coordinate, Rudin-Keisler projects $I\otimes J$ onto $J$ and therefore $\fin^{\otimes\al+1}$ onto $\fin^{\otimes\al}$. 
             For limit $\al$,
             suppose that  for every $m\leq n<\omega$, $\pi_{n,m}:\omega^{\alpha_n}\rightarrow\omega^{\alpha_m}$ is a Rudin-Keisler projection of $\fin^{\alpha_n}$ onto $\fin^{\alpha_m}$. Fix any $N<\omega$ Let us define $f:\omega^\alpha\rightarrow\omega^{\alpha_N}$ by applying 
             $$f(\l k,x\r)=\begin{cases}
                 f_{k,N}(x) & k\geq N\\
                 a^* &k<N
             \end{cases}$$
             where $a^*$ is any fixed element of $\omega^{\alpha_N}$.
             Now  if $Y\subseteq \omega^{\alpha_N}$,  then $f^{-1}[Y]=\cup_{n\geq N}\{n\}\times f^{-1}_{n,N}[Y]$ and $f^{-1}_{n,N}[Y]$. If $Y\in \fin^{\alpha_N}$ then $f^{-1}_{n,N}[Y]\in \fin^{\times\alpha_n}$ and therefore $f^{-1}[Y]\in \fin^{\otimes\alpha}$. If $Y\notin \fin^{\otimes\alpha_N}$, then $f^{-1}_{n,N}[Y]\notin \fin^{\otimes\alpha_n}$ and therefore $f^{-1}[Y]\notin \fin^{\alpha}$.             
Since $\leq_{RK}$ is transitive, we conclude thee lemma.
\end{proof}

        There is a simple characterization of the Tukey type of $\fin^{\otimes\al}$ given in the following theorem:
        \begin{theorem}
            For every $1<\al<\omega_1$, 
$(\fin^{\otimes\al},\subseteq)\equiv_{T}\omega^{\omega}.$
            
        \end{theorem}
        \begin{proof}
            By induction on $\al$. For $\al=2$, this is Proposition \ref{Prop: equivalence with omega product}. For successor $\al$, by Proposition \ref{Prop: equivalence with omega product},
            $$\fin^{\otimes\al+1}=\fin\otimes\fin^{\otimes\al}\equiv_T\fin\times \prod_{n<\omega}\fin^{\otimes\al}.$$
            By the induction hypothesis, $\fin^{\otimes\al}\equiv_T\omega^\omega$, and $\fin\equiv_T\omega$. Therefore, by Fact \ref{fact: taking power of omega},
            $$\fin\times \prod_{n<\omega} \fin^{\otimes\alpha}\equiv_T\omega\times\prod_{n<\omega}\omega^\omega\equiv_T\omega\times \omega^\omega\equiv_T\omega^\omega.$$
            So we conclude that $\fin^{\otimes\alpha+1}\equiv_T\omega^\omega$.
            For limit $\al$, we have by Theorem \ref{Theorem: basic bound for product} that
            $$\fin^{\otimes\al}=\sum_{\fin}\fin^{\otimes\al_n}\leq_T\omega\times \prod_{n<\omega}\omega^\omega\equiv_T\omega^\omega.$$
            For the other direction, we have by the previous lemma that $\omega^\omega\equiv_T\fin^{\otimes 2}\leq_T \fin^{\otimes\al}$, as desired.
        \end{proof}

        \begin{definition}
            We say that an ultrafilter $U$ on $\omega^\alpha$ is {\em $\alpha$-rapid} if $\pi_*(U)$ is rapid. where $\pi$ is the projection to the first coordinate.
        \end{definition}
        It is clear that if $U$ is $\alpha$-rapid, then $\omega^\omega\leq_T \pi_*(U)\leq_{RK}U$; hence we have the following:
        \begin{corollary}\label{corollary: sufficient for FINalpha}
            Suppose that $U$ is an $\alpha$-rapid ultrafilter over $\omega^\alpha$ such that $\fin^{\otimes\alpha}\subseteq U^*$ and $\fin^{\otimes\alpha}$-p.i.p.\ holds.
            Then $U\cdot U\equiv_T U$. 
        \end{corollary}
        \begin{proof}
By Corollary \ref{Cor: Sufficient condition for product} for $I=\fin^{\otimes\alpha}$, it remains to verify that $\prod_{n<\omega}\fin^{\otimes\alpha}\leq_T U$. Indeed, $\prod_{n<\omega}\fin^{\otimes\alpha}\equiv_T\prod_{n<\omega}\omega^\omega$ and therefore by Fact \ref{fact: taking power of omega}, $\prod_{n<\omega}\fin^{\otimes\alpha}\equiv_T\omega^\omega$. Since $U$ is $\alpha$-rapid, $\omega^\omega\leq_T\pi_*(U)\leq_{RK} U$ and therefore $\prod_{n<\omega}\fin^{\otimes\alpha}\leq_T U$. It follows that $U\cdot U\equiv_T U$. 
        \end{proof}


The following fact that 
each
$\fin^{\otimes\al}$
is a $p^+$-ideal
is well-known
(see \cite{BST22}, \cite{DobrinenJSL15}, \cite{DobrinenJML16}, \cite{Kurilic15}), and included here for self-containment.

        \begin{proposition}
            Suppose that $\l A_i\mid i<\omega\r$ is a decreasing sequence of sets in $(\fin^{\otimes\alpha})^+$.
            Then there is an $A\in (\fin^{\otimes\alpha})^+$ such that for every $i<\omega$, $A\setminus A_i\in \fin^{\otimes\alpha}$.
        \end{proposition}
            
        \begin{proof}
            By induction on $\alpha$. For $\alpha=1$, $\fin$ is indeed a $p^+$-ideal. Suppose that $\fin^{\otimes\alpha}$ has proven to be a $p^+$-ideal, and let $\l A_i\mid i<\omega\r\subseteq (\fin^{\otimes\alpha+1})^+$ be a  decreasing sequence. We may assume that each $A_i$ is in standard form, namely, for every $n<\omega$,  either $(A_i)_n=\emptyset$ or $(A_i)_n\in (\fin^{\otimes\alpha})^+$. 
            Let $$A=\bigcup_{i<\omega}\{(\pi''A_i)(i)\}\times (A_i)_{(\pi''A_i)(i)}$$
            First we note that $A\in (\fin^{\otimes\alpha+1})^+$. To see this, note that since the $A_i's$ are decreasing then whenever $i<j$:
            \begin{enumerate}
                \item $\pi''A_j\subseteq \pi''A_i$.
                \item  For each $n<\omega$, $(A_j)_n\subseteq(A_i)_n$.
            \end{enumerate}
            It follows that for $i<j<\omega$, $(\pi''A_j)(j)\in \pi''A_i$ and $(\pi''A)(j)=(\pi''A_j)(j)>(\pi''A_i)(i)=(\pi''A)(i)$. So $\{(\pi''A_i)(i)\mid i<\omega\}\in \fin^+$ and for each $i$, $(A)_{(\pi''A)(i)}=(\pi''A_i)_{(\pi''A_i)(i)}\in (\fin^{\otimes\alpha})^+$. To see that $A\setminus A_i\in \fin^{\otimes\alpha}$, for each $i\leq j$,  $$(A)_{(\pi''A_j)(j)}(A_j)_{(\pi''A_j)(j)}\subseteq (A_i)_{(\pi''A_j)(j)}.$$ We conclude that $A\setminus A_i\subseteq \cup_{j<i}(\{(\pi''A_j)_j\}\times (A_j)_{(\pi''A_j)(j)}\in \fin^{\otimes \alpha}$. At limit steps, $\delta$ then the proof is completely analogous.
        \end{proof}
        \begin{corollary}
            Let $G$ be $P(\omega^\alpha)/\fin^{\otimes\alpha}$-generic over $V$.
            Then $G$ satisfies $\fin^{\otimes\alpha}$-p.i.p.
        \end{corollary}
        \begin{lemma}
           If $G$ is $P(\omega^\alpha)/\fin^{\otimes\alpha}$-generic  over $V$, then $G$ is $\alpha$-rapid.
        \end{lemma}
        \begin{proof}
            Let $f:\omega\rightarrow\omega$ be any function in $V[G]$. By $\sigma$-closure of $P(\omega^\alpha)/\fin^{\otimes\alpha}$, $f\in V$. We proceed by a density argument, let $X\in P(\omega^\alpha)/\fin^{\otimes\alpha}$, shrink $X$ to $X_1\in P(\omega^\alpha)/\fin^{\otimes\alpha}$ so that $X_1$ is in standard form. By definition of $(\fin^{\otimes\alpha})^+$, $\pi''X_1$ is infinite and so we can shrink $\pi''X_1$ to $Y_1$, still infinite such that $Y_1(n)\geq f(n)$.
            Define $X_2=\cup_{n\in Y_1}\{n\}\times (X_1)_n$. Since $X_1$ was in standard form, for each $n\in Y_1$, $(X_1)_n$ is positive, and so, $X_2\in (\fin^{\otimes\alpha})^+$. Note that $X_2\subseteq X$ and $\pi''X_2=Y_1$. By density there is $X\in G$ such that  for every $n<\omega$, $(\pi''X)(n)\geq f(n)$ and therefore $\pi_*(G)$ is rapid, namely $G$ is $\alpha$-rapid.
        \end{proof}
        As corollary we obtain the following theorem:
        \begin{theorem}\label{thm.finalpha}
            Suppose that $G$ is a $P(\omega^\alpha)/\fin^{\otimes\alpha}$-generic ultrafilter over $V$. Then $G\cdot G\equiv_TG$.
        \end{theorem}
         \begin{proof}
        We proved that $\fin^{\otimes\alpha}\subseteq G^*$, $G$ satisfies $\fin^{\otimes\alpha}$-p.i.p.\ and that $G$ is $\alpha$-rapid. So by Corollary \ref{corollary: sufficient for FINalpha}, $G\cdot G\equiv_T G$.
        \end{proof}

Recalling Theorem \ref{Theorem: Natasha's result JSL}, 
$G_{\al}<_T\mathcal{U}_{\text{top}}$ and $G_{\al}$ is not basically generated for each $\al<\om_1$.
    The point is that although the complexity of the generic ultrafilter $G_{\al}$ increases with $\alpha$, it still satisfies $G\cdot G\equiv_TG<_T\mathcal{U}_{\text{top}}$

\section{Milliken-Taylor Ultrafilters}

In this section, we prove that Milliken-Taylor ultrafilters have the same Tukey type as their Fubini product. 
Milliken-Taylor ultrafilters 
are  ultrafilters on base set 
 $\FIN:=[\om]^{<\om}\setminus \{\emptyset\}$ which witness instances of Hindman's Theorem \cite{Hindman}.
They are the  analogues of Ramsey ultrafilters on the base set  $\FIN$,
 but they are not Ramsey ultrafilters, nor even p-points, as shown by Blass   in \cite{Blass87}, where they were called {\em stable ordered union ultrafilters}.
These ultrafilters have been widely investigated (see for instance \cite{Eisworth02} and \cite{Mildenberger11}).

We now define Milliken-Taylor ultrafilters, using notation from \cite{TodorcevicBK}.
For $n\le\infty$, $\FIN^{[n]}$ denotes the set of block sequences in $\FIN$ of length $n$, where a {\em block sequence}  is a sequence 
$\langle x_i: i<n\rangle \sse \FIN$ such $i<j<n$ implies $\max(x_i)<\min(x_j)$.
For $n<\om$ and a block sequence $X=\langle x_i: i<n\rangle\in \FIN^{[n]}$, $[X]=\{\bigcup_{i\in I} x_i: I \sse n\}$.
For an infinite block sequence $X=\langle x_i: i<\om\rangle\in \FIN^{[\infty]}$, 
$$
        [X]=\{\bigcup_{i\in I} x_i: I \in \FIN\}
 $$
For $X,Y\in \FIN^{[\infty]}$, define $Y\le X$ iff $[Y]\sse [X]$.
Given $X\in \FIN^{[\infty]}$ and $m\in\om$, $X/m$ denotes $\langle x_i:i\ge n\rangle$ where $n$ is least such that $\min(x_n)>m$.
    Define $Y\le^*X$ iff there is some $m$ such that $[Y/m]\sse [X]$.
    Related definitions for finite block sequences are similar.

\begin{definition}\label{def.FINIdeal}
    An ultrafilter $U$ on base set $\FIN$ is {\em Milliken-Taylor} iff 
    \begin{enumerate}
        \item
    For each $A\in U$, there is an infinite block sequence $X\in \FIN^{[\infty]}$
such that $[X]\sse A$ and $[X]\in U$; and 
\item 
For each sequence $\langle X_n:n<\om\rangle$ of block sequences such that $X_0\ge^* X_1\ge^*\dots$ and each $[X_n]\in U$,
there is a diagonalization $Y\in \FIN^{[\infty]}$ such that  $[Y]\in U$ and $X_n\ge^* Y$ for each $n<\om$.
\end{enumerate}
\end{definition}

  Thus, a
Milliken-Taylor ultrafilter $U$ has
$\{A\in U:\exists X\in\FIN^{[\infty]}\, (A=[X])\}$
as  a filter base, and such a filter base
has the property that 
 almost decreasing sequences have diagonalizations. 
In this sense, Milliken-Taylor ultrafilter behave like p-points even though, technically, they are not.
The following ideal corresponds to  property $(2)$:

\begin{definition}
    Let $I$ be the set of all $X\subseteq \FIN$ such that for some $N\in \om$, $\forall x\in X, x\cap N\neq\emptyset$.
\end{definition}
\begin{claim}
    $I$ is an ideal and $I\subseteq U^*$ for each Milliken-Taylor ultrafilter $U$.
    \end{claim}
   \begin{proof}
        Clearly, $\emptyset\in I$ and $I$ is downwards closed. To see that $I$ is closed under unions, let   $X,Y\in I$   and let $N_X, N_Y\in \om$ witness this.
        Then $\max(N_X,N_Y)$ witnesses that $X\cup Y\in I$. By condition $(1)$, every $A\in U$ contains $[X]$ for some infinite block sequence $X$; in particular, $A\notin I$. 
 \end{proof}
\begin{proposition}
    If $Y\leq^* X$ then $[Y]\setminus [X]\in I$.
\end{proposition}
\begin{proof}
    If $Y\leq^* X$ then there is $m$ such that $[Y/m]\subseteq [X]$ and so every element $b\in [Y]\setminus [X]$ must be a finite union of sets which includes some element below $m$. 
\end{proof}\begin{proposition}
    $U$ satisfies $I$-p.i.p.
\end{proposition}

\begin{proof}
    Let $\l A_n\mid n<\omega\r\subseteq U$. Then by property $(1)$ of $U$, we can shrink each $A_n$ to $[X_n]\in U$ such that $X_{n+1}\leq X_n$. By property $(2)$ there is $[X]\in U$ such that $[X]\leq^* [X_n]$ for every $n$. Thus $[X]\setminus [X_n]\in I$ and in particular $[X]\setminus A_n\in I$.
\end{proof}
\begin{proposition}
    $I\equiv_T\omega$.
\end{proposition}
\begin{proof}
    Define $f:\omega\rightarrow I$ by $f(n)=\{x\in \FIN\mid x\cap n\neq\emptyset\}$.  Clearly $f$ is monotone, and by definition of $I$, $f$ is cofinal. It is also clear that $f$ is unbounded since $\cup_{n\in A}f(n)=\FIN$, whenever $A\in[\omega]^{\omega}$ (and $I$ in a proper ideal).
\end{proof}

The projection maps  $\min$  and $\max$ from $\FIN$ to $\om$ are clear.
Given a Milliken-Taylor ultrafilter $U$,
let $U_{\min}$ and $U_{\max}$ denote the Rudin-Keisler projections  of $U$ according to $\min$ and $\max$, respectively. 
Blass showed in \cite{Blass87} that $U_{\min}$ and $U_{\max}$ are both Ramsey ultrafilters.  
Hence,  it follows that 
$\omega^\omega\leq_T U_{\min}\leq_{RK}U$, and therefore we have the following corollary:

\begin{corollary}
     $\prod_{n<\omega}I\leq_T U$.
\end{corollary}

\begin{theorem}\label{thm.MillikenTaylor}
    Suppose that $U$ is a  Milliken-Taylor ultrafilter.  Then $U\cdot U\equiv_TU$.
\end{theorem}

\begin{proof}
    We proved that if $U$ is 
    Milliken-Taylor, then for the ideal $I$, we have that $I\subseteq U^*$, $I$-p.i.p., and $\prod_{n<\omega}I\leq_T U$. By Corollary \ref{Cor: Sufficient condition for product}, we conclude that $U\cdot U\equiv_T U$.
\end{proof}

We conclude this section with a short proof  that the $\min$-$\max$ projection of $U$ is Tukey equivalent to its Fubini product with itself.
The map $\min$-$\max:\FIN\ra\om\times\om$  is defined by $\min$-$\max(x)=(\min(x),\max(x))$, for 
$x\in \FIN$.  Let 
$U$ be a Milliken-Taylor ultrafilter and let 
$U_{\min,\max}$ denote the ultrafilter on $\om\times\om$  which is the $\min$-$\max$ Rudin-Kesiler projection of $U$.
Blass showed in \cite{Blass87} that $U_{\min,\max}$ is isomorphic to $U_{\min}\cdot U_{\max}$ and hence,
$U_{\min,\max}$
is not a p-point.
Dobrinen and Todorcevic showed in \cite{Dobrinen/Todorcevic11} that $U_{\min,\max}$ is not a q-point, but is rapid, and that, assuming CH, $U_{\min}$ and $U_{\max}$ are Tukey strictly below $U_{\min,\max}$ which is Tukey strictly below $U$.
It follows from the proof of 
Theorem 72 in \cite{Dobrinen/Todorcevic11} 
that $U_{\min,\max}$ has appropriately defined diagonalizations and hence,  has the $J$-p.i.p.\ for the ideal $J=\{\min$-$\max(A): 
 A\in I\}\sse U_{\min,\max}^*$ on $\om\times\om$; hence  the work in this paper implies the following theorem.
However, we give a shorter proof by combining results from  \cite{Blass87} and \cite{Dobrinen/Todorcevic11}.

\begin{corollary}
If $U$ is a Milliken-Taylor ultrafilter, then 
$$U_{\min,\max}\equiv_T 
U_{\min,\max}\cdot U_{\min,\max}.$$
\end{corollary}

\begin{proof}
By results of Blass in \cite{Blass87},
$U_{\min,\max}\cong U_{\min}\cdot U_{\max}$, and both $U_{\min}$
and $U_{\max}$ are Ramsey ultrafilters.
For rapid p-points $U,V$, a result in
\cite{Dobrinen/Todorcevic11} showed that   $U\cdot V\equiv_T V\cdot U$, and hence,
$$
(U\cdot V)\cdot (U\cdot V)
\equiv_T U\cdot (V\cdot V)\cdot U
\equiv_T U\cdot V\cdot U
\equiv_T U\cdot U\cdot V
\equiv_T U\cdot V.
$$
The corollary follows.
\end{proof}

\begin{remark}
The theorems in this section should generalize to Milliken-Taylor ultrafilters on $\FIN^{[\infty]}_k$ as well as their Rudin-Keisler projections, as their diagonalization properties will imply the $I$-p.i.p.\ for the naturally associated ideal $I$.
\end{remark}
        
\section{Further directions and open questions}
\begin{question}
    Is it a $ZFC$ theorem that for any two ultrafilters $U,V$ over $\omega$,   $U\cdot V\equiv_T V\cdot U$?
\end{question}
For $\kappa$-complete ultrafilters over measurable cardinals $\kappa$, this is indeed the case, as was proved by the authors in  \cite{TomNatasha}. However, the proof essentially uses the well foundedness of the ultrapower by a $\kappa$-complete ultrafilter $U$.

    A natural strategy to answer the previous question would be to take $U$ such that $U<_TU\cdot U$. The only constructions of  ultrafilters $U$ such that $U<_TU\cdot U$ ensure that $U\not\geq_T \omega^\omega$. By the results of this paper we can generate examples where $U\not\geq_T \prod_{n<\omega} I$ for some ideal $I$ such that $U$ is $I$-p.i.p. 
    
    Using such $U$, we need to find an ultrafilter $V$ such that $U\cdot V\not\equiv_T V\cdot U$. We know that following hold:
    $$U\cdot V\equiv_TU\times V\cdot V, \ \ \ V\cdot U\equiv_T V\times U\cdot U$$
So natural assumptions would be to require that $V\equiv_TV\cdot V$, and in order for $V$ not to interfere with the assumption $U<_TU\cdot U$, in order to have $V\leq_T U$. This guarantees that
$$U\cdot V\equiv_TU<_T U\cdot U\equiv_T V\cdot U$$
However, the assumptions above are not consistent since if $V\cdot V\equiv_T V$ then $V\geq_T \omega^\omega$, and therefore if $V\leq_T U$ then also $U\geq_T \omega^\omega$.  This leads to the following question:
\begin{question}
    Is it consistent that there are two ultrafilters $U,V$ such that $V\equiv_T V\cdot V\leq_T U<_TU\cdot U$? Or more precisely, is the class of ultrafilters which are Tukey reducible to their Fubini product upwards closed with respect to the Tukey order?
\end{question}
It seems that  the Tukey type of $\omega^\omega$ plays an important role in the calculations of the Tukey type of $U\cdot U$:
\begin{question}
    Is there an ultrafilter $U$ such that $\omega^\omega\le_T U<_T U\cdot U$? 
\end{question}
\begin{question}
    Is it consistent to have an ultrafilter $U$ such that $U$ is not rapid but $U\geq_T \omega^\omega$? What about $U$ which is a $p$-point?
\end{question}
\begin{question}
Is there a $p^+$-ideal $I$ on a countable set $X$ such that some $P(X)/I$ generic ultrafilter $U$ is Tukey-top?
\end{question}

\begin{question}
    Is it true that for every $p^+$-ideal $I$ on a countable set $X$, a generic ultrafilter $U$ on $P(X)/I$ satisfies $U\cdot U\equiv_TU$?
\end{question}

\bibliographystyle{amsplain}
\bibliography{ref}
\end{document}